\documentclass[11pt,a4paper,final]{article}
\usepackage[latin1]{inputenc}
\usepackage{amsfonts}
\usepackage{amscd}
\usepackage{amsmath}
\usepackage{theorem}
\usepackage{mathrsfs}
\usepackage{graphicx}
\usepackage{fancybox,amssymb}
\usepackage{geometry}
\usepackage[english]{babel}
\usepackage{epstopdf}
\usepackage{caption}
\usepackage{subcaption}
\usepackage{bbding}

\newcommand{\mF}{\mathcal{F}}

\newcommand{\R}{\mathbb{R}}

\newcommand{\mP}{\mathbb{P}}
\newcommand{\mE}{\mathbb{E}}
\newcommand{\E}{\mathbb{E}}
\newcommand{\md}{\,{\rm d}}

\newcommand{\s}{\sum\limits}

\newcommand{\one}{{\mathbbm{1}}}
\usepackage{bbm}
\newcommand{\dW}{{\md W}}
\newcommand{\dB}{{\md B}}
\newcommand{\du}{{\md u}}
\newcommand{\dv}{{\md v}}
\newcommand{\ds}{{\md s}}
\newcommand{\dtildet}{{\md \tilde{t}}}

\newcommand{\exptauzeta}{\exp\left(-U_{\tau_\zeta^{x-l}}\right)}

\newcommand{\eexptauzetax}{\mE_r\left[\exp\left(-U_{\tau_\zeta^{x}}\right)\right]}
\newcommand{\tauzeta}{\tau_\zeta}

\newcommand{\Ito}{{It\^{o}}}

\theoremstyle{break}
\theorembodyfont{}
\newtheorem{Def}{Definition}[section]
\newtheorem{Remark}[Def]{Remark}
\newtheorem{Lemma}[Def]{Lemma}

\newtheorem{Theorem}[Def]{Theorem}
\newtheorem{Example}[Def]{Example}
\newtheorem{Assumption}[Def]{Assumption}

\usepackage{color}
\usepackage{xcolor}

\makeatletter
\newenvironment{proof}{\noindent{\textit{Proof:}}}{%
\unskip\nobreak\hfil\penalty50\hskip1em\null\nobreak
$\Box$
\parfillskip=\z@\finalhyphendemerits=0\endgraf\bigskip}

\let\oldendBsp\endBsp
\def\endBsp{\unskip\nobreak\hfil\penalty50\hskip1em\null\nobreak\hfil%
$\blacksquare$\parfillskip=\z@\finalhyphendemerits=0\endgraf\oldendBsp}
\let\oldendBem\endBem
\def\endBem{\unskip\nobreak\hfil\penalty50\hskip1em\null\nobreak\hfil%
$\blacksquare$\parfillskip=\z@\finalhyphendemerits=0\endgraf\oldendBem}
\makeatother

\date{July 2021}

\title{Two Approaches for a Dividend Maximization Problem under an Ornstein-Uhlenbeck Interest Rate}

\author{ Julia Eisenberg$^1$ \hspace{1.0em} Stefan Kremsner$^2$   \hspace{1.0em}   Alexander Steinicke$^3$   \hspace{1.0em}   }

\begin{document}
\maketitle
\begin{abstract}\noindent
We investigate a dividend maximization problem under stochastic interest rates with Ornstein-Uhlenbeck dynamics. This setup also takes negative rates into account.
First a deterministic time is considered, where an explicit separating curve $\alpha(t)$ can be found to determine the optimal strategy at time $t$. In a second setting we introduce
a strategy-independent stopping time. The properties and behavior of these optimal control problems in both settings are analyzed in an analytical HJB-driven approach as well as using backward stochastic differential equations.
\end{abstract}

\vspace{6pt}
\settowidth\labelwidth{{\it 2010 Mathematical Subject Classification: }}%
                \par\noindent {\it 2010 Mathematical Subject Classification: }%
                \rlap{Primary}\phantom{Secondary}
                93E20\newline\null\hskip\labelwidth
                Secondary 49L20, 60H30

{\noindent
	\footnotetext[1]{Institute of Mathematical Methods in Economics, Vienna University of Technology, Wiedner Hauptstra\ss e 8, 1040 Wien, Austria. \hspace*{1.5em}  {\tt  jeisenbe{\rm@}fam.tuwien.ac.at} \Envelope}
	\footnotetext[2]{Department of Mathematics, University of Graz, Heinrichstra\ss{}e 36, 8010 Graz, Austria. \hspace*{1.5em}
		{\tt stefan.kremsner{\rm@}uni-graz.at}}
	\footnotetext[3]{ Department of Mathematics and Information Technology, Montanuniversit�t Leoben, Peter Tunner-Stra\ss e 25/I, 8700 Leoben, Austria. \hspace*{1.5em}  {\tt  alexander.steinicke{\rm@}unileoben.ac.at}}}

\section{Introduction}
Maximizing expected discounted dividends is a popular and well investigated topic in insurance mathematics, whose roots go back to the work of Bruno De Finetti in 1957 \cite{definetti}. Modeling the surplus of a company by a Brownian motion with drift, Cramer-Lundberg or a general L\'evy process, putting constraints in form of value at risk, time inconsistent preferences or incorporating a random funding -- these problems have been studied, for instance, in \cite{asmussen}, \cite{azcue}, \cite{strini}, \cite{hailiang}. A general overview of the existing literature can be found in \cite{altho}, \cite{avanzi}, \cite{schmidli}. 

In several cases, e.g.\ in \cite{asmussen}, \cite{azcue}, \cite{JiaPi}, the optimal dividend strategy was proved to be a constant barrier or a band strategy. However, in some settings the optimal strategy could not be determined or even be shown to exist. Then, typically, a viscosity solution approach has to be used instead. 

An important feature of almost every setting is that one considers the value of expected discounted dividends up to some time of ruin due to the chosen dividend strategy. The dependence of the time horizon on the strategy makes the solution of the problem quite complicated, particularly in the presence of a stochastic interest rate. Several papers consider the setting with a stochastic interest rate, e.g.\ \cite{soner}, \cite {JiaPi}, \cite{tiziano}. 
\\For instance, letting the interest rate be given by a mean-reverting Ornstein-Uhlenbeck (Vasicek model) process leads to a two dimensional control problem, which cannot be easily solved via the corresponding Hamilton--Jacobi--Bellman equation. The optimal solution seems to be of barrier type where the barrier is given by a highly non-linear function depending on the interest rate.
Therefore, it is hard to explicitly solve such a problem or even calculate the return function corresponding to a non-linear barrier. An attempt to tackle this problem was done for instance in \cite{eisenberg}. However, the value function could not be shown to be sufficiently smooth, and a viscosity solution approach was applied. 

In the present paper, we consider the dividend maximization problem in a different setup. We modify the usual setting and consider two different model scenarios. Under the assumption of an Ornstein-Uhlenbeck interest rate and Brownian surplus, the ruin time of an ex-dividend surplus process is neglected. Instead, we introduce a finite deterministic time horizon in the first model, and a strategy-independent but surplus-dependent stochastic time horizon in the second model. In both models, one faces a 3-dimensional control problem. We demonstrate two solution approaches: solving the corresponding Hamilton--Jacobi--Bellman (HJB) equation and a backward stochastic differential equation (BSDE) approach. Whilst in the first model, we are able to calculate the value function and the optimal strategy explicitly via the HJB approach, in the second model only the BSDE approach allows us to show that the value function is smooth enough. 
\\Both approaches, applied in a parallel way, demonstrate their advantages and disadvantages. If one is able to calculate a return function that is sufficiently smooth and solves the HJB equation, the HJB approach leads to stronger results than the BSDE approach. However, the calculation of a candidate return function can be extremely time and space-consuming. Instead, the BSDE approach allows to calculate the value function and the optimal strategy numerically. 

BSDEs were first introduced by Bismut \cite{bismut1978introductory} and later extensively studied in non-linear form by Pardoux and Peng \cite{pardoux1990adapted}. The concept of BSDEs has proved itself very useful, especially in the context of finance and stochastic optimal control, see for example 
\cite{el1997backward, ma1999forward, kohlmann2000relationship, pham, kremsner2020deep}. Moreover for high-dimensional problems, there exist numerical algorithms for BSDEs which do not suffer from a curse of dimensionality (see for example \cite{han2018solving} or the survey on BSDE numerics \cite{chessari2021numerical}), whereas classical HJB-related methods like finite differences to solve the PDE associated to the stochastic control problem may be very inefficient. 

The remainder of the paper is structured as follows. In Section 2, we solve the problem via a HJB and a BSDE approach in a setting with finite time $T > 0$. In Section 3, we consider the same  problem with a stochastic surplus-dependent time horizon. We illustrate both sections with numerical examples.
\section{Dividend Maximization with a Deterministic Time Horizon}\label{sec:deterministic_time}
In the following, we consider an insurance company whose surplus is given by a Brownian motion with drift $X_t=x+\mu t+\sigma W_t$, where $\{W_t\}$ is a standard Brownian motion on a probability space $\left(\Omega,\mathcal{F},\mathbb{P}\right)$. The considered insurance company is allowed to pay out dividends, where the accumulated dividends until $t$ are given by $C_t := \int_0^tc_s\md s$, yielding for the post dividend surplus $X^c$:
\[
X_t^c=x+\mu t+\sigma W_t-C_t\;.
\]
Let further $\{B_t\}$ be a standard Brownian motion independent of $\{W_t\}$, generating the filtration $\{\mF^B_t\}$, augmented with the $\mathbb{P}$-null sets $\mathcal{N}$.
We let the underlying filtration $\{\mF_t\}$ be the filtration generated by $\{W_t,B_t\}$, also augmented with the probability space's $\mathbb{P}$-null sets $\mathcal{N}$. Moreover, we assume that $\mathcal{F}$ is the completed sigma-algebra generated by $\{W_t,B_t\}$.
In the following, we will use the common convention $\mE_{t,y}[.]=\mE[.|Y_t=y]$ for some process $\{Y_t\}$.
\\We let the dividends be discounted by an Ornstein-Uhlenbeck process ($a,\delta>0$, $b \in \R$):
\[
r_t=re^{-a t}+b(1-e^{-a t})+\delta e^{-at}\int_0^t e^{au}\md B_u\;,
\] 
or as an SDE
\[
\md r_t=a(b-r_t)\md t+\delta \md B_t\;.
\]
Further, we let
\[
U_{t}^s:=\int_t^s r_u\md u,\; t \le s\;.
\]
We allow only strategies $c=\{c_t\}$ with $c_t\in[0,\xi]$ for some given, fixed real number $\xi>0$. Such a  strategy is called admissible if additionally $c$ is adapted to $\{\mF_t\}$. The set of admissible strategies will be denoted by $\mathfrak A$.
\smallskip
\\
In this section, we consider a deterministic time horizon $T\in(0,\infty)$. Differently than in the classical setting we do not stop our considerations once the surplus process ruins. 
As a risk measure we consider the value of expected discounted dividends and define the return function corresponding to some admissible strategy $c=\{c_s\}$ to be
\begin{align*}
&V^c(t,r,x)=\mE_{t,r,x}\Big[\int_t^{T} e^{-U_{t}^s} c_s\md s+e^{-U_{t}^T}X_{T}^c\Big]\;,
\end{align*} 
The value function is then given by
\begin{align*}
&V(t,r,x)=\sup\limits_{c\in\mathfrak A} V^c(t,r,x),\quad (t,r,x)\in[0,T]\times\R\times\R\;,
\\&V(T,r,x)=x, \quad (r,x)\in\R\times \R\;.
\end{align*}
\subsection{HJB approach}

The heuristically derived HJB equation corresponding to the problem (for details see for instance \cite{schmidli}) is
\begin{align}
V_t+\mu V_x+\frac{\sigma^2}2V_{xx}+a(b-r) V_r+\frac{\delta^2}2 V_{rr}-rV+\sup\limits_{0\le c\le\xi}c\{1-V_x\}=0\;.\label{hjb:1}
\end{align}

\subsubsection{Payout on the maximal rate}

In order to get a feeling how the optimal strategy might look like, we first consider the return function corresponding to the strategy ``always pay out on the maximal rate $ \xi$'', i.e.\ $c_s\equiv \xi$ for all $s$.

A simple calculation yields
\[
U_{t}^s=\int_t^s r_u\md u
=\frac{r_t-r_s}a+b (s-t)+\frac{\delta}a (B_s-B_t)\;.
\]
In order to investigate some further properties of the value function, we consider the moment generating function of $U_{t}^s$. 
Using an elementary change of measure technique (see for instance Schmidli \cite[p.\ 216]{schmidli}) or using the formula from Brigo and Mercurio \cite[p.\ 59]{brigo2007interest}, and letting $\tilde b:=b-\frac{\delta^2}{2a^2}$ one obtains
\begin{align}
M(s-t,r)&:=\mE_{t,r}[e^{-U_{t}^s}]\nonumber
\\&=e^{-\tilde b (s-t)}\exp\Big\{\frac{\tilde  b-r}{a}  (1-e^{-a(s-t)})-\frac{\delta^2}{4a^3}(1-e^{-a(s-t)})^2\Big\}\;. \label{U}
\end{align}
Then, one obtains immediately
\begin{align}
V^\xi(t,r,x)&=  \xi \int_t^{ T} M(s-t,r)\md s
+M(T-t,r)\mE_{t,x}\big[X_{T}^\xi\big]\nonumber
\\&=\xi \int_t^{ T} M(s-t,r)\md s
+M(T-t,r)\big\{x+(\mu-\xi)(T-t)\big\}\;,\label{Vxi}
\end{align}
where $X^\xi_t=x+(\mu-\xi) t+\sigma W_t$.
\begin{Remark}
Note that the exponent of the function $M$ defined in \eqref{U} can be split into a linear and an exponential part. For the exponential part it obviously holds, that
\[
\exp\Big\{\frac{\tilde  b-r}{a}  (1-e^{-a(s-t)})-\frac{\delta^2}{4a^3}(1-e^{-a(s-t)})^2\Big\}\le 
\begin{cases}
1&\mbox{: $\tilde b\le r$}
\\e^{\frac{\tilde b-r}a}&\mbox{: $\tilde b> r$}.
\end{cases}
\]
Thus, if $\tilde b\le 0$ and $T$ is big enough, the function $M$ will exponentially increase with the growing time, implying a strong and long-lasting negative interest rate environment. 
\end{Remark}
\begin{Assumption}\label{assumption}
In the following, we assume $\tilde b>0$.
\end{Assumption}
\bigskip
The question arises whether the strategy ``pay out on the maximal rate'' might be optimal on some time interval independent of the values $(r,x)$. This question will be answered in the next section where a candidate for the value function will be inserted into the HJB equation \eqref{hjb:1}.
\subsubsection{Derivation of the value function}
The form of the return function $V^\xi$ corresponding to the strategy ``always pay out on the maximal rate'' hints that the value function may also be linear in $x$. Therefore, we assume that there is a function $G(t,r)$ such that a candidate for the value function is given by
\begin{align}
v(t,r,x):=G(t,r)+x\mE_{t,r}[e^{-U_{T;t}}]=G(t,r)+xM(T-t,r)\;.\label{conjecture}
\end{align}
We need then to check whether $v$ solves the HJB equation \eqref{hjb:1}. In representation \eqref{conjecture}, the factor determining the optimal strategy, i.e.\ to pay or to wait, is given by $M(T-t,r)$ and depends just on $r$ and $T-t$. If $M(T-t,r)>1$ it is optimal to wait and if $M(T-t,r)<1$ then it is optimal to pay on the maximal possible rate $\xi$. 
\\For the convenience of explanations, we will use the following notation for sufficiently smooth functions $f(t,r)$ on $[0,T]\times \R$:
\[
\mathcal L(f)(t,r):= f_t(t,r)+a(b-r) f_r(t,r)+\frac{\delta^2}2 f_{rr}(t,r)-rf(t,r)\;.
\]
\begin{Lemma}
$M(u,r)$ solves the following differential equation
\begin{equation}
-M_u(u,r)+a\big( b-r\big)M_r(u,r)+\frac{\delta^2}2M_{rr}(u,r)-rM(u,r)=0\;.\label{diffF}
\end{equation}
with the boundary conditions $M(0,r)=1$, $\lim\limits_{r\to\infty}M(u,r)=0$ and $\lim\limits_{r\to-\infty}M(u,r)=\infty$ for $u> 0$.
\end{Lemma}
\begin{proof}
The proof is straightforward.
\end{proof}
If $v$ in \eqref{conjecture} is indeed the value function then the function $G$ should fulfil
\[
\mathcal L(G)(t,r)+\mu M(T-t,r)+\sup\limits_{0\le c\le \xi} c\{1-M(T-t,r)\}=0\;.
\]
Our next step is to find the function $G$ solving the above differential equation. For this purpose, we have to investigate the properties of the function $M$ in order to get rid of the supremum expression in the differential equation above.
\subsubsection{Properties of the function $M$}
In this section, we investigate the properties of the function $M$. Recall from \eqref{U} that it holds
\begin{align*}
M(T-t,r)&=\mE_{t,r}\big[e^{U_{t}^T}\big]
\\&=\exp\Big\{-\tilde b(T-t)+\frac{\tilde b-r-\frac{\delta^2}{2a^2}}a\big(1-e^{-a(T-t)}\big)+\frac{\delta^2}{4a^3}\big(1-e^{-2a(T-t)}\big)\Big\}\;.
\end{align*}
It is immediately clear that $M$ is strictly decreasing in $r$ on $[0,T)\times \R$.
\\We are searching for a curve $\alpha(t)$ such that for $(t,r)$ with $r>\alpha(t)$ it holds $M(T-t,r)<1$ and for $r< \alpha(t)$ it holds $M(T-t,r)>1$. Consider the exponent of $M$ and let
\[
\ln\big(M(T-t,r)\big)=-\tilde b(T-t)+\frac{\tilde b-r-\frac{\delta^2}{2a^2}}a\big(1-e^{-a(T-t)}\big)+\frac{\delta^2}{4a^3}\big(1-e^{-2a(T-t)}\big)\;.
\]

Obviously, $\ln(M(T-t,r))$ is strictly decreasing in $r$ for $t\in [0,T]$. 
Solving the equation $\ln(M(T-t,r)) = 0$ for $r$ guides us to define
\begin{align}
\alpha(t):&=-a\tilde b\frac{T-t}{1-e^{-a(T-t)}}+\frac{\delta^2}{4a^2}\big(1+e^{-a(T-t)}\big) +\tilde b-\frac{\delta^2}{2a^2} \nonumber
\\&=\tilde b\Big(1-a\frac{T-t}{1-e^{-a(T-t)}}\Big)-\frac{\delta^2}{4a^2}\big(1-e^{-a(T-t)}\big)\;.\label{alpha}
\end{align}
The curve $\alpha$ is uniquely defined. Due to $M_r<0$ on $[0,T)\times \R$, $\alpha$ is separating the sets 
\begin{align}
S_1:=\{(t,r):\ M(T-t,r)>1\}\;\; \mbox{and}\;\; S_2:=\{(t,r):\ M(T-t,r)<1\}.\label{S}
\end{align}
\begin{Remark}
The following properties hold true:
\begin{itemize}
\item $\alpha(T)=0$, $\alpha'(T)=ba/2>0$.
\item $\frac{T-t}{1-e^{-a(T-t)}}$ is decreasing in $t$.
\item If $\tilde b> 0$, then because $-a\frac{T-t}{1-e^{-a(T-t)}}+1\le 0$ it holds $\alpha(t)< 0$ for all $t\in[0,T)$.
\item Since $\tilde b> 0$, the function $\alpha$ is strictly increasing in $t$ with $\alpha(t)\in[\alpha(0),0]$.
\end{itemize}
\end{Remark} 
\begin{Example}
In Figure \ref{fig1:total}, we see the function $\alpha$ for different parameters. On the right picture, the curve is increasing with $\tilde b>0$. The left picture with $\tilde b<0$ shows a curve that decreases first and increases close to the time horizon $T=5$. 
\\As we assume $\tilde b>0$, see Assumption \ref{assumption}, the curve given on the left side of Figure \ref{fig1:total} is impossible in our setting. 
\end{Example}

\begin{figure*}[t]
	\centering
	\begin{subfigure}[t]{0.4\textwidth}
		\centering
		\scalebox{0.75}{
			\includegraphics[height=2.8in]{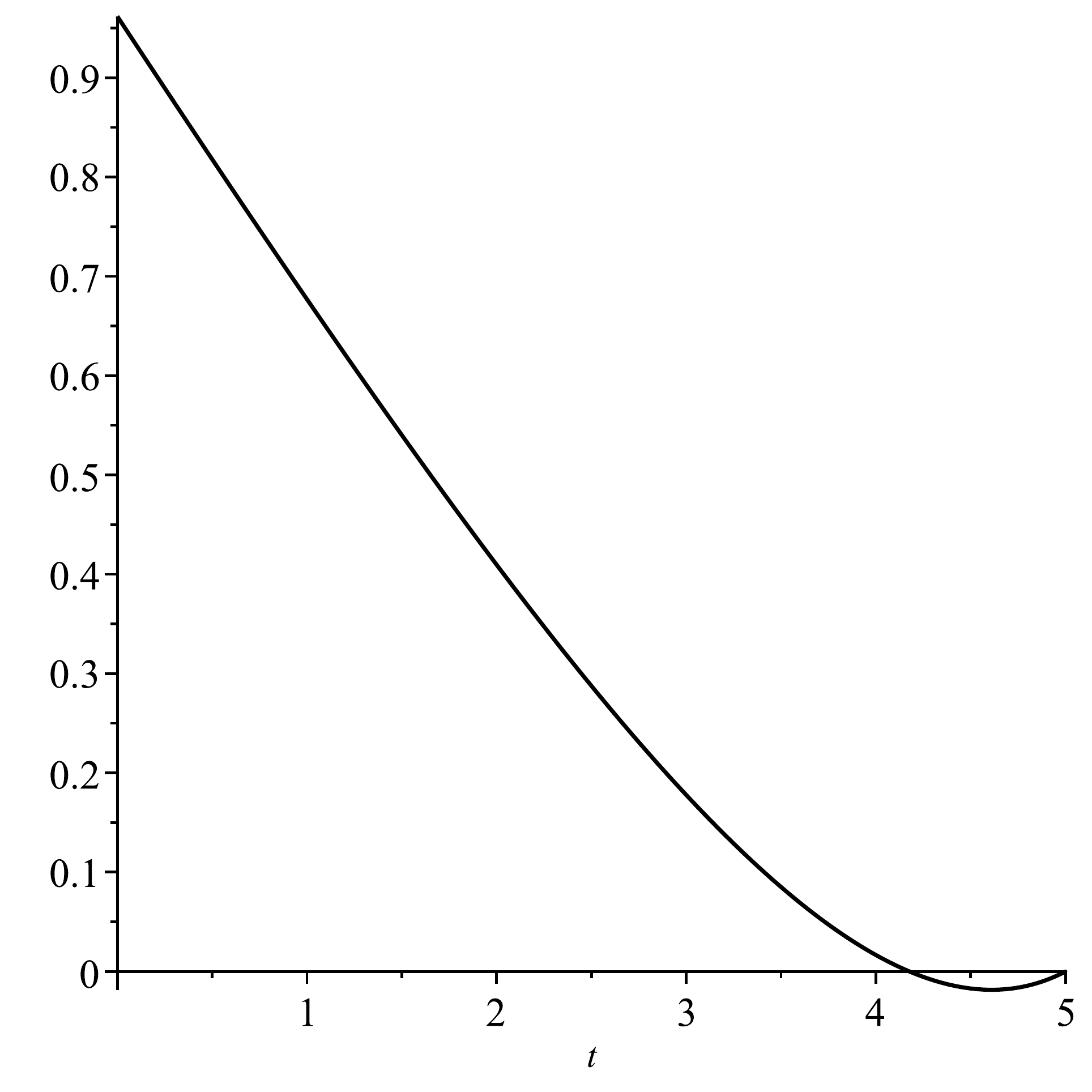}\label{fig:1}}
	\end{subfigure}%
	~ 
	\begin{subfigure}[t]{0.4\textwidth}
		\centering
		\scalebox{0.75}{
			\includegraphics[height=2.8in]{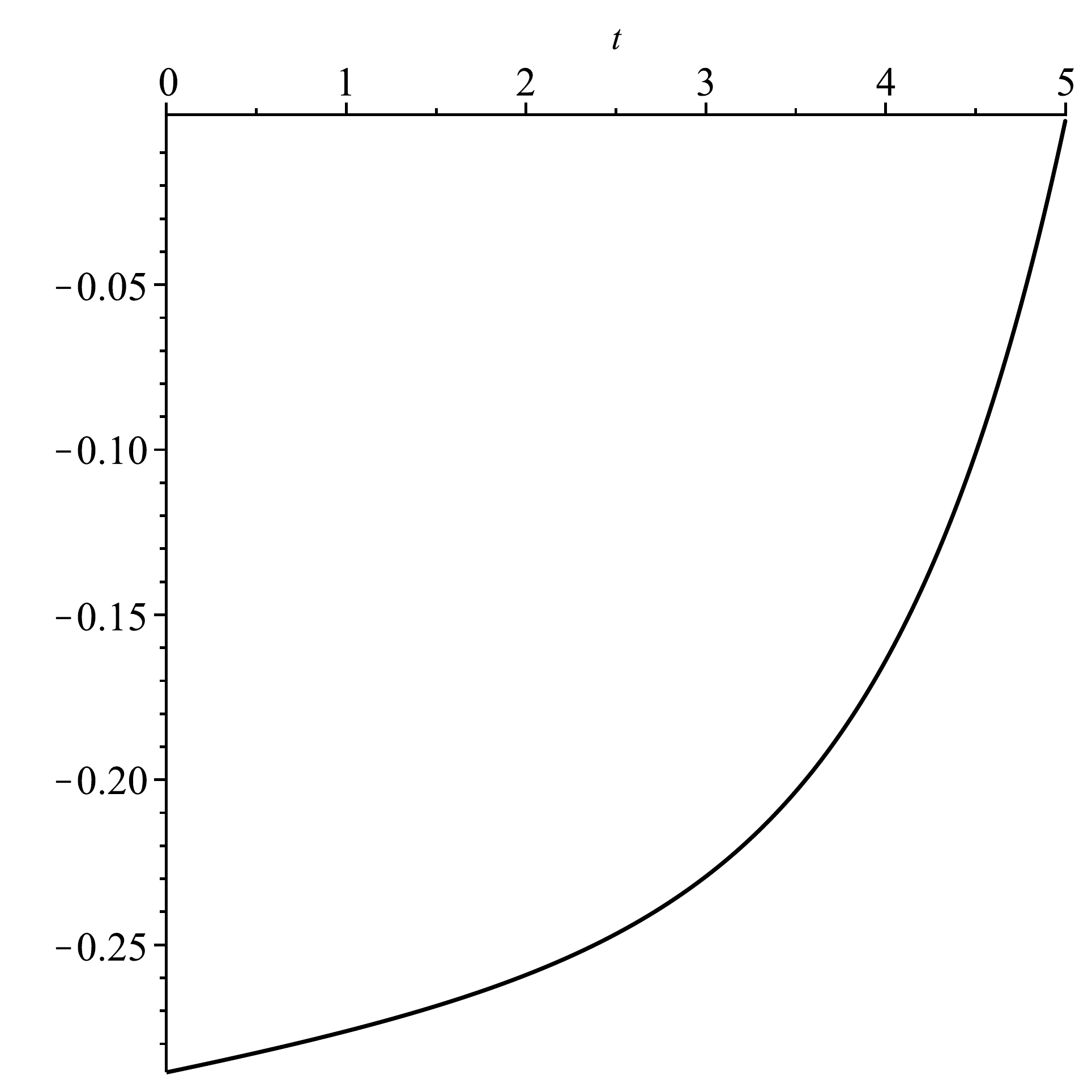}\label{fig:2}}
	\end{subfigure}
	\caption{The curve $\alpha(t)$ for $a=\delta=1$, $T=5$, $b=0.2$ (left picture) and $b=0.51$ (right picture).}
	\label{fig1:total}
\end{figure*}
\noindent
Now, knowing the function $\alpha$ we can investigate the properties of the remaining function $G$ from \eqref{conjecture}.
\subsubsection{Properties of the function $G$}
In order to find a function $G$ such that $v$ in \eqref{conjecture} solves the HJB equation \eqref{hjb:1} we define an auxiliary function $v(t,r)=xM(T-t,r)+\tilde G(t,r)$ with
\begin{align}
\tilde G(t,r):=\mE_{t,r}\bigg[\int_t^T e^{-U_{t}^s}\cdot \xi\one_{[r_s>\alpha(s)]}\md s +e^{-U_{t}^T}\int_t^T \mu -\xi\one_{[r_s>\alpha(s)]}\md s\bigg]\label{defG}\;,
\end{align}
i.e.\ $v$ is the return function corresponding to the strategy $c_s=\xi\one_{[r_s>\alpha(s)]}$. 
Our target is to show that the function $\tilde G(t,r)$ solves the differential equation
\begin{align}
\begin{split}\label{tildeG}
\mathcal L(\tilde G)(t,r)+\mu M(T-t,r)+\xi\one_{[r> \alpha(t)]}\big(1-M(T-t,r)\big)=0 
\end{split}
\end{align}
with boundary conditions $\tilde G(T,r)=0$, $\lim\limits_{r\to\infty}\tilde G(t,r)=0$ and $\lim\limits_{r\to-\infty}\tilde G(t,r)=\infty$.
\\Letting
\begin{equation}
\gamma(t,r):= \mE_{t,r}\bigg[\int_t^T \big\{e^{-U_{t}^s}-e^{-U_{t}^T}\big\}\cdot \one_{[r_s>\alpha(s)]}\md s \bigg]\label{gamma}
\end{equation}
we can rewrite the function $\tilde G$ as follows
\begin{align*}
\tilde G(t,r)&=\mu (T-t) M(T-t,r)+ \xi\mE_{t,r}\bigg[\int_t^T \big\{e^{-U_{t}^s}-e^{-U_{t}^T}\big\}\cdot \one_{[r_s>\alpha(s)]}\md s \bigg]
\\&=\mu (T-t) M(T-t,r)+ \xi \gamma(t,r)\;.
\end{align*}
\begin{Remark}\label{rem:M}
Note that $M(u,r)$ solves Differential equation \eqref{diffF}. Therefore, for $\mu (T-t) M(T-t,r)$ one obtains
\[
\mathcal L\Big(\mu(T-t)M(T-t,r)\Big)=-\mu M(T-t,r)
\]
The function $\mu (T-t)M(T-t,r)$ attains zero at $t=T$, $\lim\limits_{r\to\infty}\mu(T-t)M(T-t,r)=0$ and $\lim\limits_{r\to-\infty}\mu(T-t)M(T-t,r)=\infty$. Moreover, it holds  $\mu(T-t)M(T-t,\alpha(t))=\mu(T-t)$.
\end{Remark}
\begin{Remark}
Consider the return function $V^\xi$ given in \eqref{Vxi} corresponding to the strategy ``always pay out on the maximal rate''. 
In the same way like in Remark \ref{rem:M}, one can show that $\xi\int_t^T M(T-s,r)\md s$ solves $\mathcal L\Big(\xi\int_t^T M(T-s,r)\md s\Big)=-\xi \big(1-M(T-t,r)\big)$. It means that $V^\xi$ solves
\[
\mathcal L(V^\xi)(t,r)+\mu M(T-t,r)+\xi\big(1-M(T-t,r)\big)=0
\]
on $[0,T] \times \R$.
\\We now see that for $r<\alpha(t)$, the function $V^\xi$ does not solve HJB equation \eqref{hjb:1}.
\end{Remark}
In the following, we concentrate on the function $\gamma$ given in \eqref{gamma}. Due to Remark \ref{rem:M} we need to show that $\gamma$ fulfils 
\begin{align}
\mathcal L(\gamma)(t,r)+\one_{[r> \alpha(t)]}\big(1-M(T-t,r)\big)=0\;. \label{eq:gamma}
\end{align}
\begin{Lemma}\label{lem:repr}
The function $\gamma$ defined in \eqref{gamma} can be written as 
\[
\int_0^{T-t} \int_{\alpha(u+t)}^\infty M(u,r)\cdot \big(1-M(T-u-t,z)\big)\cdot \varphi(z,u,r)\md z\md u,
\]
where $M$ is given in \eqref{U} and
\begin{align}
\begin{split}\label{phi}
\varphi(z,u,r):=&\frac1{\sqrt{2\pi \frac{\delta^2}{2a}(1-e^{-2au})}}
\\&\times \exp\left\{-\frac{\big(z-re^{-au}-b(1-e^{-au})+\frac{\delta^2}{2a^2}(1-e^{-au})^2\big)^2}{2\frac{\delta^2}{2a}(1-e^{-2au})}\right\}\;.
\end{split}
\end{align}
\end{Lemma}
\begin{proof}
$\bullet$ Using Tonelli's theorem and the law of total probability, we get
\begin{align}
\gamma(t,r)&=\mE_{t,r}\bigg[\int_t^T \big\{e^{-U_{t}^s}-e^{-U_{t}^T}\big\}\cdot \one_{[r_s>\alpha(s)]}\md s \bigg] \nonumber
\\&= \int_t^T \mE_{t,r}\big[\big\{e^{-U_{t}^s}-e^{-U_{t}^T}\big\}\cdot \one_{[r_s>\alpha(s)]}\big]\md s \nonumber
\\&= \int_t^T \int_{\alpha(s)}^\infty \mE_{t,r}\big[\big\{e^{-U_{t}^s}-e^{-U_{t}^s-U_s^T}\big\}|r_s=z\big]\cdot \mP_{t,r}[r_s\in\md z]\md s \nonumber
\\&= \int_t^T \int_{\alpha(s)}^\infty \mE_{t,r}\big[e^{-U_{t}^s}|r_s=z\big]\cdot\big(1-\mE\big[e^{-U_s^T}|r_s=z\big]\big)\cdot \mP_{t,r}[r_s\in\md z]\md s \nonumber
\\&= \int_t^T \int_{\alpha(s)}^\infty \mE_{t,r}\big[e^{-U_{t}^s}|r_s=z\big]\cdot\big(1-M(T-s,z)\big)\cdot \mP_{t,r}[r_s\in\md z]\md s\;.\label{lastline}
\end{align}
$\bullet$ The term $\mP_{t,r}[r_s\in\md z]$, inside the above integrals, is the density  of the random variable $r_{s-t}$. Since, $r_{s-t}$ is normally distributed with mean $\beta(s-t,r):=re^{-a(s-t)}+b(1-e^{-a(s-t)})$ and variance $\frac{\delta^2}{2a}(1-e^{-2a(s-t)})$, see \cite[p.\ 522]{BorSa}, it holds that
\[
\mP_{t,r}[r_s\in\md z]=\frac1{\sqrt{2\pi}\frac{\delta^2}{2a}(1-e^{-2a(s-t)})}\exp\bigg\{-\frac{\big(z-\beta(s-t,r)\big)^2}{2\frac{\delta^2}{2a}(1-e^{-2a(s-t)})}\bigg\}\;.
\]
$\bullet$ Consider now the first factor inside the integrals. Formula 1.8.7(1) in \cite[p.\ 525]{BorSa} along with $(1-e^{-2a(s-t)})=(1-e^{-a(s-t)})(1+e^{-a(s-t)})$ yield
\begin{align*}
\mE_{t,r}\Big[e^{-U_{t}^s}|r_s=z\Big]
&=e^{-\tilde b(s-t)}\exp\bigg\{-\frac{z+r-2b+\frac{\delta^2}{a^2}}a\cdot \frac{1-e^{-a(s-t)}}{1+e^{-a(s-t)}}\bigg\}
\\&=M(s-t,r)\cdot \exp\bigg\{-2\cdot \frac{z-\beta(s-t,r)}{2\frac{\delta^2}{2a}(1-e^{-2a(s-t)})}\cdot \frac{\delta^2}{2a^2}(1-e^{-a(s-t)})^2\bigg\}
\\&\quad {}\times \exp\bigg\{-\Big(\frac{\delta^2}{2a^2}\Big)^2\frac{(1-e^{-a(s-t)})^4}{2\frac{\delta^2}{2a}(1-e^{-2a(s-t)})}\bigg\}\;.
\end{align*}
$\bullet$ Then, completing the square gives
\begin{align*}
\mE_{t,r}\Big[e^{-U_{t}^s}|r_s=z\Big]\cdot \mP_{t,r}[r_s\in\md z]
=M(s-t,r)\cdot\varphi(z,s-t,r)\;.
\end{align*}
$\bullet$ Changing the variable $u=s-t$ in \eqref{lastline} yields the desired result. 
\end{proof}
\begin{Lemma}\label{u=0}
The function $\gamma$ defined in \eqref{gamma} fulfils $\gamma\in\mathcal C^{1,2}([0,T)\times\R)$.
\end{Lemma}
\begin{proof}
Recall that $\varphi$ is the density of a normal distribution, see \eqref{phi}. 
Let further 
\[
\Delta(u,r):=r^{-au}+b(1-e^{-au})-\frac{\delta^2}{2a^2}(1-e^{-au})^2\;.
\]
Then, 
\[
\lim\limits_{u\to 0}\frac{z-\Delta(u,r)}{\sqrt{\frac{\delta^2}{2a}(1-e^{-2a u})}} =
\begin{cases}
\infty &\mbox{: $z>r$}\\
-\infty &\mbox{: $z<r$}\\
0 &\mbox{: $z=r$}\;.
\end{cases}
\]
Changing the variable by letting $y=\frac{z-\Delta(u,r)}{\sqrt{\frac{\delta^2}{2a}(1-e^{-2a u})}}$ yields 
\small{
\begin{align*}
&\lim\limits_{u\to 0}\int_{\alpha(t+u)}^\infty \varphi(z,u,r)\cdot \big(1-M(T-t-u,z)\big)\md z
\\&= \lim\limits_{u\to 0}\int_{\tfrac{\alpha(t+u)-\Delta(u,r)}{\sqrt{\delta^2(1-e^{-2a u})/(2a)}}}^\infty \frac {e^{-y^2/2}}{\sqrt{2\pi}}\cdot \left(1-M\left(T-t-u,\;y\sqrt{\frac{\delta^2}{2a}(1-e^{-2a u})}+\Delta(u,r)\right)\right)\md y
\\&=\begin{cases}
1-M(T-t,r) &\mbox{: $\alpha(t)>r\;,$}\\
0 &\mbox{: $\alpha(t)\le r\;.$}
\end{cases}
\end{align*}
}
Using similar arguments, the representation of $\gamma$ given in Lemma \ref{lem:repr} and the Leibniz integral rule yields the claim. 
\end{proof}
\begin{Lemma}\label{lem:diff}
The function $\gamma$ defined in \eqref{gamma} solves Differential equation \eqref{eq:gamma}.
\end{Lemma}
\begin{proof}
$\bullet$ Recall from Lemma \ref{lem:repr} that $\gamma$ can be written as 
\[
\int_0^{T-t} \int_{\alpha(u+t)}^\infty M(u,r)\cdot \big(1-M(T-u-t,z)\big)\cdot \varphi(z,u,r)\md z\md u 
\]
with $M$ given in \eqref{U} and $\varphi$ given in \eqref{phi}. Lemma \ref{u=0} yields $\gamma\in \mathcal C^{1,2}([0,T)\times \R)$. 
\medskip
\\$\bullet$ It is straightforward to build derivatives of $\varphi$ and show that
\[
\varphi_t(z,t,r)-a\big\{ b-r-\frac{\delta^2}{a^2} (1-e^{-at})\big\}\varphi_r(z,t,r)-\frac{\delta^2}2\varphi_{rr}(z,t,r)=0\;.
\]
$\bullet$ Recall that $M(u,r)$ solves Differential equation \eqref{diffF}.\medskip
\\$\bullet$ Using $M(0,r)=M(T-t,\alpha(t))=1$ and $M_r(u,r)=(1-e^{-au})/a$, we conclude
\begin{align*}
&&\gamma_t(t,r)&= \int_0^{T-t} M(u,r) \int_{\alpha(u+t)}^\infty \varphi(z,u,r)\cdot M_t(T-u-t,z)\md z\md u\;,
\\&& \gamma_r(t,r)&= \int_0^{T-t} M_r(u,r)\int_{\alpha(u+t)}^\infty \varphi(z,u,r)\cdot (1-M(T-u-t,z))\md z\md u
\\&&&\quad {}+\int_0^{T-t} M(u,r)\int_{\alpha(u+t)}^\infty \varphi_r(z,u,r)\cdot (1-M(T-u-t,z))\md z\md u\;,
\\&& \gamma_{rr}(t,r)&= \int_0^{T-t} M_{rr}(u,r)\int_{\alpha(u+t)}^\infty \varphi(z,u,r)\cdot (1-M(T-u-t,z))\md z\md u
\\&&&\quad {}-2\int_0^{T-t} M(u,r)\int_{\alpha(u+t)}^\infty \frac{1-e^{-au}}a\varphi_r(z,u,r)\cdot (1-M(T-u-t,z))\md z\md u
\\&&&\quad {}+\int_0^{T-t} M(u,r)\int_{\alpha(u+t)}^\infty \varphi_{rr}(z,u,r)\cdot (1-M(T-u-t,z))\md z\md u\;.
\end{align*}
$\bullet$ Note that $M_t(T-u-t,z)=M_u(T-u-t,z)$ and $\mathcal L(M)(u,r)=0$. Therefore, we can conclude that
\small
\begin{align*}
\mathcal L(\gamma)(t,r)&= \int_0^{T-t} M(u,r) \int_{\alpha(u+t)}^\infty \varphi(z,u,r)\cdot M_u(T-u-t,z)\md z\md u
\\&\quad {}+\int_0^{T-t} \mathcal L(M)(u,r)+M_u(u,r) \int_{\alpha(u+t)}^\infty \varphi(z,u,r)\cdot \big(1-M(T-u-t,z)\big)\md z\md u
\\&\quad {}+\int_0^{T-t} M(u,r) \int_{\alpha(u+t)}^\infty \Big(a\big\{ b-r-\frac{\delta^2}{a^2} (1-e^{-at})\big\}\varphi_r(z,t,r)+\frac{\delta^2}2\varphi_{rr}(z,t,r)\Big) 
\\& \hspace{7.5cm}\times \big(1-M(T-u-t,z)\big)\md z\md u
\\& = \int_0^{T-t} \partial_u\bigg( M(u,r) \int_{\alpha(u+t)}^\infty \varphi(z,u,r)\cdot \big(1-M(T-u-t,z)\big)\md z\bigg)\md u\;.
\end{align*}
\normalsize
Consequently, we can get rid of the $\md u$-integral and get
\begin{align*}
\mathcal L(\gamma)(t,r)&= M(T-t,r) \int_{\alpha(T)}^\infty \varphi(z,T-t,r)\cdot \big(1-M(0,z)\big)\md z
\\&\quad{}-M(0,r)\lim\limits_{u\to 0} \int_{\alpha(t+u)}^\infty \varphi(z,u,r)\cdot \big(1-M(T-t-u,z)\big)\md z\;.
\end{align*}
The proof of Lemma \ref{u=0} gives
\[
\mathcal L(\gamma)(t,r)=-\Big(1-M(T-t,r)\Big)\one_{[\alpha(t)>r]}\;,
\]
which corresponds to Differential equation \eqref{eq:gamma}.
\end{proof}
Now, we are ready to prove the verification theorem.
\begin{Theorem}[Verification Theorem]
The function $v=xM(T-t,r)+\tilde G(t,r)$, with $M$ given in \eqref{U} and 
\begin{align*}
\tilde G(t,r)&=\mu (T-t) M(T-t,r)
\\&\quad{}+ \xi\int_0^{T-t} \int_{\alpha(u+t)}^\infty M(u,r)\cdot \big(1-M(T-u-t,z)\big)\cdot \varphi(z,u,r)\md z\md u 
\end{align*}
with $\varphi$ in \eqref{phi} and $\alpha$ in \eqref{alpha}, is the value function, solves the HJB equation \eqref{hjb:1}. The optimal strategy is given by $c^*=\{c^*_s\}$ with $c^*_s=\xi\one_{[r_s>\alpha(s)]}$ 
\end{Theorem}
\begin{proof}
Remark \ref{rem:M} and Lemma \ref{lem:diff} prove that $v$ solves HJB equation \eqref{hjb:1}.
\\
Let $c$ be an arbitrary admissible strategy. Then, using Ito's formula one has
\begin{align*}
e^{-U_{t}^s}v(t,r_t,X^c_t)&=v(s,r,x)+\int_s^t e^{-U_{s}^y}\big\{\mathcal L(v)(y,r_y)+(\mu-c_y)M(T-y,r_y)\big\}\md y
\\&\quad{}+ \delta \int_s^t e^{-U_{s}^y} v_{r}\md B_y+\sigma \int_s^t e^{-U_{s}^y} v_{x}\md W_y\;.
\end{align*}
Since the stochastic integrals are martingales (due to {\Ito} isometry) with expectation zero, and $v$ solves HJB equation \eqref{hjb:1}, we can conclude 
\begin{align*}
\mE_{s,r,x}\big[e^{-U_{s}^t}v(t,r_t,X^c_t)\big]&\le v(s,r,x) -\mE_{s,r,x}\Big[\int_s^t e^{-U{s}^y}c_y\md y\Big]\;.
\end{align*}
Letting $t\to T$ yields then 
\[
\mE_{s,r,x}\big[e^{-U_{s}^T}X^c_T\big]+\mE_{s,r,x}\Big[\int_s^T e^{-U_{s}^y}c_y\md y\Big]\le v(s,r,x)\;,
\]
which proves our claim.
\end{proof}
We conclude that the optimal strategy does not depend on the surplus. This is due to the fact that we do not stop our considerations at the time of ruin, i.e.\ at the time when the surplus hits zero. The penalty for having a negative surplus is reflected only in the expected lump sum payment at $T$, which is given by $x+\mu T-\int_0^T c_s\md s$. 
\\The decision to pay or to wait is a feedback strategy of a current interest rate. Given the curve $\alpha$ defined in \eqref{alpha}, if the interest rate at time $t$ lies above $\alpha(t)$, it is optimal to pay on the maximal possible rate and not to pay otherwise. 
\smallskip
\\The economic interpretation is as follows. Assumption $\tilde b>\frac{\delta^2}{2a^2}$ implicates that $\alpha(t)\le 0$, and the mean-reverting OU process being below zero will be pushed up. If the interest rate is below $\alpha(t)$, i.e.\ negative, then the expected discounting factor is increasing in time. 
\\Thus, if the interest rate stays under $\alpha(t)$ the discounting factor attains its maximum at $T$. Since, $\mu t>\mu-\int_0^T c_s\md s$ for any strategy $c=\{c_s\}$, it is not surprising that for $r_s<\alpha(s)$ one should not pay dividends until $T$.   
\\On the other hand, if the interest rate $r_t$ lies above the curve $\alpha(t)$, excursions of $r_t$ into the positive half-line will lead to an decreasing in time discounting factor. Therefore, one would be rather willing to pay immediately on the maximal rate than to wait until $T$.

\subsection{BSDE approach}\label{sec:bsdesec1}


In this section we will tackle the problem of finding the optimal strategy $c$ for
\[ V^c(t,r,x)=\mathbb{E}_{t,r,x}\left[\int_t^T e^{-U_t^s}c_s\ds+e^{-U_t^T}X^c_T\right]\]
by an ansatz using a generalized Hamiltonian and resulting coupled forward-backward SDEs (FBSDEs), elaborated for example in \cite[Section 4.2]{carmona}, \cite[Section 6.4.2]{pham} and \cite[Section 10.1.1]{touzi}.
Note that the dynamics of the process $X^c$ is
$$dX^c_s =(\mu-c_s)\ds+\sigma \dW_s+ 0\dB_s\;,\text{ for }s\geq t,\text{ and } X^c_t=x\;.$$
Therefore, the generalized Hamiltonian $H\colon\Omega\times[0,T]\times\R\times\R\times\R^{1\times 2}\to\R$ takes the form
\[H(s,\mathtt{x},c,y,z)=(\mu-c)y+\mathrm{tr}\left(\begin{pmatrix}
	\sigma & 0
\end{pmatrix}^T z\right)+e^{-U_t^s}c
\]
and actually does not depend on $\mathtt{x}$.
The function $g\colon \Omega\times\R\to\R$ is given by 
\[
g(\mathtt{x})=e^{-\int_t^Tr_u \du}\mathtt{x} = e^{-U_t^T} \mathtt{x}.
\]
Both ${H}$ and $g$ are affine in $(\mathtt{x},c)$ resp. $\mathtt{x}$, hence they are concave. The maximum can be achieved by setting $c$ to
\begin{align}\label{eq: cstrat}\hat c=\arg\max_cH(s,\mathtt{x},c,y,z)=\begin{cases}\xi,& e^{-U_t^s}-y>0, \\0, & e^{-U_t^s}-y\leq 0. \end{cases}\end{align}
With the above expressions we are able to put up the corresponding BSDE, 
\[Y_s=
\partial_{\mathtt{x}}
g(X^{\hat c}_T)+\int_s^T
\partial_{\mathtt{x}}
H(u,X^{\hat c}_u,\hat c_u,Y_u,Z_u)\du-\int_s^TZ_u\begin{pmatrix}
	\dW_u\\\dB_u
\end{pmatrix},\quad t\leq s\leq T,\]
taking a particularly easy form without generator here:
\[Y_s=e^{-U_t^T}-\int_s^TZ_{u,1}\dW_u-\int_s^TZ_{u,2}\dB_u\;.\]
As this BSDE's terminal condition only depends on $B$, it follows that $Z_{u,1}=0$ and that the solution $Y$ is given by $Y_s=\E_{t,r,x}\left[ e^{-\int_t^Tr_u \du}\middle|\mathcal{F}^B_s\right]=\E_{t,r,x}\left[ e^{-U_s^T}\middle|\mathcal{F}^B_s\right]$, which we will calculate in the following. Note that 
\[Y_s=\E_{t,r,x}\left[ e^{-\int_t^Tr_u \du}\middle|\mathcal{F}^B_s\right]=e^{-\int_t^sr_u \du}\E_{t,r,x}\left[ e^{-\int_s^Tr_u \du}\middle|\mathcal{F}^B_s\right]=e^{-U_t^s}\E_{t,r,x}\left[ e^{-U_s^T}\middle|\mathcal{F}^B_s\right]\]
By the SDE for the $r$ process, we may write 
\begin{align}\label{eq: SDEr}
	\int_s^Tr_v \dv = \frac{r_s-r_T}{a}+\frac{\delta}{a}(B_T-B_s)+b(T-s)\;,\end{align}
and can hence compute
\begin{align}\label{eq: condexplicit}
	&\E_{t,r,x}\left[ e^{-U_s^T}\middle|\mathcal{F}^B_s\right]=\E_{t,r,x}\left[e^{-\int_s^Tr_u \du}\middle|\mathcal{F}^B_s\right]=\E_{t,r,x}\left[e^{\frac{r_T-r_s}{a}-\frac{\delta}{a}(B_T-B_s)-b(T-s)}\middle|\mathcal{F}^B_s\right]\nonumber\\
	&=e^{-\frac{r_s}{a}-b(T-s)}\E_{t,r,x}\left[e^{\frac{r_T}{a}-\frac{\delta}{a}(B_T-B_s)}\middle|\mathcal{F}^B_s\right]=e^{-\frac{r_s}{a}-b(T-s)}\E_{t,r,x}\left[e^{\frac{r_T}{a}-\frac{\delta}{a}(B_T-B_s)}\middle|\mathcal{F}^B_s\right]\nonumber\\
	&=e^{-\frac{r_s}{a}-b(T-s)}\E_{t,r,x}\left[e^{\frac{1}{a}\left(r_se^{-a(T-s)}+b\left(1-e^{-a(T-s)}\right)+\delta e^{-aT}\int_s^Te^{au}\dB_u\right)-\frac{\delta}{a}(B_T-B_s)}\middle|\mathcal{F}^B_s\right]\nonumber\\
	&=e^{-\frac{r_s}{a}\left(1-e^{-a(T-s)}\right)-b(T-s)+\frac{b}{a}\left(1-e^{-a(T-s)}\right)}\E_{t,r,x}\left[e^{\frac{\delta}{a}\int_s^T(e^{a(u-T)}-1)\dB_u}\middle|\mathcal{F}^B_s\right]\nonumber\\
	&=e^{-\frac{r_s}{a}\left(1-e^{-a(T-s)}\right)-b(T-s)+\frac{b}{a}\left(1-e^{-a(T-s)}\right)}\E_{t,r,x}\left[e^{\frac{\delta}{a}\int_s^T(e^{a(u-T)}-1)\dB_u}\right]\nonumber\\
	&=e^{-\frac{r_s}{a}\left(1-e^{-a(T-s)}\right)-b(T-s)+\frac{b}{a}\left(1-e^{-a(T-s)}\right)+\frac{\delta^2}{2a^2}\int_s^T(1-e^{a(u-T)})^2\du}\;,
\end{align} 
where we used the explicit form of $r_T$, independent increments of Wiener integrals and the expression for the mean of a log-normal distribution. We conclude 
\begin{align*}Y_s&=e^{-\int_t^sr_u \du}e^{-\frac{r_s}{a}\left(1-e^{-a(T-s)}\right)-b(T-s)+\frac{b}{a}\left(1-e^{-a(T-s)}\right)+\frac{\delta^2}{2a^2}\int_s^T(1-e^{a(u-T)})^2\du}\\
&=e^{-\int_t^sr_u \du}e^{-\tilde b(T-s)+\frac{\tilde b-r_s-\frac{\delta^2}{2a^2}}a\big(1-e^{-a(T-s)}\big)+\frac{\delta^2}{4a^3}\big(1-e^{-2a(T-s)}\big)}\\
&=e^{-U_t^s}M(T-s,r_s)\;,
\end{align*}
with $M$ from \eqref{U}.
Using this explicit expression and \eqref{eq: cstrat} we obtain an optimal strategy $\hat{c}$ by setting for $s\geq t$,
$$\hat{c}_s=\begin{cases}\xi\;,& e^{-U_t^s}-Y_s>0, \\0, & e^{-U_t^s}-Y_s\leq 0\;. \end{cases}$$
Inserting for $Y_s$ at the barrier $e^{-U_t^s}-Y_s=0$, we see that it can be reduced to
\begin{align*}
e^{-U_t^s}-Y_s &=0 \iff\\
	\quad e^{-\int_t^sr_u du}-e^{\int_t^sr_u du}\E_{t,r,x}\left[e^{-\int_s^Tr_u du}\middle|\mathcal{F}^B_s\right]&=0 \iff \\
	\quad \E_{t,r,x}\left[e^{-\int_s^Tr_u du}\middle|\mathcal{F}^B_s\right]&=1  \iff \\
	\quad \E_{t,r,x}\left[e^{-U_s^T}\middle|\mathcal{F}^B_s\right] &= 1\;,
\end{align*}
coinciding with the results from Section 2.
Using the explicit form of the conditional expectation and taking logarithms, the above equality yields
\begin{align*}
	&0={-\tilde b(T-s)+\frac{\tilde b-r_s-\frac{\delta^2}{2a^2}}a\big(1-e^{-a(T-s)}\big)+\frac{\delta^2}{4a^3}\big(1-e^{-2a(T-s)}\big)}
\end{align*}
From there, it is easy to infer the same barrier curve $\alpha$ from \eqref{alpha}.

A numerical illustration of the interest rate compared to the curve $\alpha$, separating the strategy areas, is provided in Figure \ref{fig:num}.

\begin{figure}[t]
	\centering
	\includegraphics[scale=0.5]{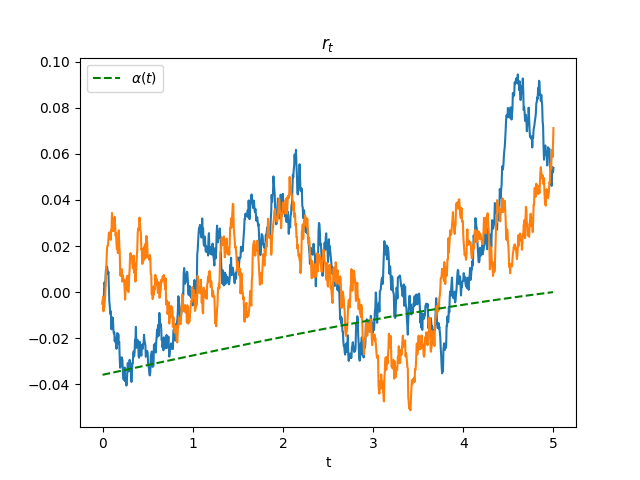}
	\caption{Two paths of $r_t$ in view of $\alpha(t)$.}
	\label{fig:num}
\end{figure}

\section{Dividend Maximization with an Exogenous Stochastic Time Horizon}\label{sec:stochastic_time}
In this section, we consider again an insurance company whose surplus is given by a Brownian motion with drift $X_t=x+\mu t+\sigma W_t$, where $\{W_t\}$ is a standard Brownian motion on a probability space $\left(\Omega,\mathcal{F},\mathbb{P}\right)$. The insurance company is paying out dividends, where the accumulated dividends until $t$ are given by $C_t=\int_0^t c_s\md s$.
\\The discounting rate is given by an Ornstein-Uhlenbeck process with the dynamics 
($a,\delta>0$, $b\in\R$):
\[
r_t=re^{-a t}+b(1-e^{-a t})+\delta e^{-at}\int_0^t e^{au}\md B_u\;,
\] 
The problem considered in Section \ref{sec:deterministic_time} is now modified by changing the time horizon. Usually, one assumes that the time horizon depends on the chosen dividend strategy. However, this approach does not seem to be entirely realistic. The analysts of the insurance company under consideration may fix the life time of the company as the ruin time corresponding to the dividend payout strategy with a fixed constant rate, say $\zeta\le \mu$. We allow just for the strategies $c=\{c_s\}$ with accumulated dividends given by $C_t=\int_0^tc_s\md s$ and $c_s\in[0,\xi]$ for some given and fixed $\xi>0$. Such a  strategy is called admissible if additionally $c$ is adapted to $\{\mF_t\}$. The set of admissible strategies will be denoted by $\mathfrak A$. 
\medskip
\\
As in Section \ref{sec:deterministic_time}, we will use $\tilde b=b-\frac{\delta^2}{2a^2}$,
\begin{align*}
	U_s:= U_0^s = \int_0^s r_u\md u\;.
\end{align*}
and recall
\begin{align*}
	M(s,r)=\mE_r[e^{-U_s}]=e^{-\tilde b s}\exp\Big\{\frac{\tilde  b-r}{a}  (1-e^{-as})-\frac{\delta^2}{4a^3}(1-e^{-as})^2\Big\}\;.
\end{align*}
In the following, we let $\md  X_t^{H,\zeta}:= (\mu-\zeta)\md t+\sigma \md W_t$, where $H$ indicates that a company whose surplus' drift exceeds $\mu-\zeta$ is considered as financially healthy. Further, $\zeta\in[0,\mu]$ is fixed and:
\begin{align*}
	&\tau_\zeta:=\inf\{t\ge 0:\ X_t^{H,\zeta}=0\}\;.
\end{align*}

\subsection{HJB approach}
For the dynamic programming principle to work we need to introduce an auxiliary process 
\begin{align*}
&L_t=-l-\zeta t\quad \mbox{and}\quad L_t^c=-l-\zeta t+\int_0^t c_s\md s\;.
\end{align*}
The process $L$ describes the difference of $X^{H,\zeta}$ and an ex-dividend process $X^c$. The initial value $l$ describes the historical  difference existing at time $0$, i.e.\ $X^{H,\zeta}_0=x-l$, $X_0=x$. In particular, for an admissible strategy $c=\{c_s\}$ it holds that
\[
X^c_{\tau_\zeta}=x+\int_0^{\tau_\zeta} \mu-c_s\md s+\sigma W_{\tau_\zeta}= \int_0^{\tau_\zeta}\zeta-c_s\md s+l=-L^c_{\tau_\zeta}\;.
\]
And, we define the target functional as
\begin{align*}
V^c(r,x,l)&:=\mE_{r,x,l}\Big[\int_0^{\tau_\zeta}e^{-U_s}c_s \md s+e^{-U_{\tau_\zeta\wedge \tau_c}}X_{\tau_\zeta }^c\Big]
\\&=\mE_{r,x,l}\left[\int_0^{\tau_\zeta }e^{-U_s}c_s \md s+e^{-U_{\tau_\zeta}}\left(l+\int_0^{\tau_\zeta}\zeta-c_s\md s\right)\right]
\\&=\mE_{r,x,l}\Big[\int_0^{\tau_\zeta }e^{-U_s}c_s \md s-e^{-U_{\tau_\zeta}}L^c_{\tau_\zeta}\Big]\;.
\end{align*}   
To ensure that the value function is well-defined we again require Assumption \ref{assumption}, i.e.\ $\tilde b>0$.  
Otherwise, by using, for instance, the constant strategy $c_s=\zeta\le \xi$ and noting that $\mP_x[\tau_\zeta =\infty]>0$ as $\zeta\le \mu$, one would get for $x>0$:
\begin{align*}
V^\zeta(r,x,l)&=\mE_{r,x,l}\Big[\int_0^{\tau_\zeta}e^{-U_s}\zeta \md s+e^{-U_{\tau_\zeta}}X_{\tau_\zeta}^\zeta\Big]
=\mE_{r,x,l}\Big[\int_0^{\tau_\zeta}e^{-\tilde b s}M(s,r)\zeta \md s\Big]=\infty\;.
\end{align*}
We are searching for the value function 
\[
V(r,x,l):=\sup\limits_{c\in\mathfrak A} V^c(r,x,l)\;.
\]
Using standard dynamic programming arguments, see for example \cite[p.\ 98]{schmidli} one can heuristically derive the following HJB equation: 
\begin{align*}
\mu V_x+\frac{\sigma^2}2V_{xx}+a(b-r)V_r+\frac{\delta^2}2 V_{rr}-r V-\zeta V_l+\sup\limits_{0\le c\le\xi}c\big(1-V_x+V_l\big)=0\;.
\end{align*}
We conjecture again that the optimal strategy is of a barrier type. In order to get an idea about the desired barrier, we consider the differential quotient of the value function with respect to $x$. For this purpose, let $h>0$ be very small and $c$ be an $\varepsilon$ admissible $(r,x+h,l+h)$-strategy, i.e.\
\[
V^c(r,x+h,l+h)+\varepsilon\ge  V(r,x+h,l+h)\;.
\]
Then, $c$ is also an admissible strategy for $(r,x,l)$. With a slight abuse of notation, we write $\tau_\zeta^{x-l}$ to indicate the starting value of the underlying process $X^{H,\zeta}$, one gets $\tau_{\zeta}^{x-l}=\tau_\zeta^{x+h-l-h}$ a.s. Therefore, we can conclude 
\begin{align*}
V(r,x+h,l+h)-V(r,x,l)&\le V^c(r,x+h,l+h)+\varepsilon -V^c(r,x,l)
\\&=h\mE_r\left[\exptauzeta\right]+\varepsilon\;. 
\end{align*}    
On the other hand, if $\tilde C$ is an $\varepsilon$ strategy for $(r,x,l)$ then it is also admissible for $(r,x+h,l+h)$, meaning that
\begin{align*}
V(r,x+h,l+h)-V(r,x,l)&\ge V^{\tilde c}(r,x+h,l+h)-\varepsilon -V^{\tilde c}(r,x,l)
\\&=h\mE_r\left[\exptauzeta\right]-\varepsilon\;. 
\end{align*} 
Since $\varepsilon$ was arbitrary, we can conclude that 
\[
\lim\limits_{h\to 0}\frac{V(r,x+h,l-h)-V(r,x,l)}{h}=\mE_r\left[\exptauzeta\right]\;.
\]
In particular, if the value function $V$ is differentiable with respect to $x$ and $l$, then one gets
\[
V_x(r,x,l)-V_l(r,x,l)=\mE_r\left[\exptauzeta\right]\;.
\]
The stopping time $\tau_\zeta^{x-l}$ is independent of $B$, and the distribution function of $\tau_\zeta^{x-l}$ is well-known, see \cite[p.\  295]{BorSa}. It means the expression $\mE_r\left[\exptauzeta\right]$ can be explicitly calculated, at least as a power series. 
\begin{Remark}
Following the path of Section \ref{sec:deterministic_time}, one should first consider $\mE_r\left[\exptauzeta\right]$ and find, if the case maybe, a curve $\theta(r,l)\neq 0$ such that $\mE_r\left[\exp\left(-U_{\tau_\zeta^{\theta(r,l)-l}}\right)\right]\equiv 1$ for all $(r,l) \in \R^2$.
\\In the second step, one calculates the return function corresponding to the strategy $c_s=\xi \one_{[x>\theta(r_s,L_s)]}$.
\\Then, if the regularity conditions are fulfilled, one can check whether the function solves the HJB equation.  
\end{Remark}
\subsubsection{Properties of $\eexptauzetax$}
Because $W$ and $B$ are independent, we can define
\begin{align*}
\phi(r,z):=&\eexptauzetax=\mE_{r}[M(\tau_\zeta^z,r)]\\
=&\mE\Big[e^{-\tilde b s}\exp\Big\{\tfrac{\tilde  b-r}{a}  (1-e^{-a \tau_\zeta^z})-\tfrac{\delta^2}{4a^3}(1-e^{-a\tau_\zeta^z})^2\Big\}\Big]\;.
\end{align*}
\noindent
Note that for $r\ge 0$ the function $M(t,r)$ is strictly decreasing in $t$. Since $\tau_\zeta^z$ is strictly increasing in $z$, we conclude that $\phi(r,z)$ is strictly decreasing in $z$, yielding
\[
\phi(r,z)< \phi(r,0)=1
\]
for $(r,z)\in\R^2$.
\\This means, in particular, that the desired curve lies in $\{r<0\}$.
\noindent
In the next section we consider the case $r<0$. 
\subsubsection{The case $r<0$}
For the sake of clarity, we itemize the properties of the function $\phi$ below if $r<0$.\bigskip
\\
$\bullet$ It is easy to see that
\begin{align*}
\phi(r,0)=1 \quad \mbox{and}\quad \lim\limits_{z\to \infty}\phi(r,z) =0\;.
\end{align*}
$\bullet$ It is hard to derive the properties of the function $\phi$ from its expectation representation. 
To calculate the expectation, one can consider the function $M(r,t)$ first and write it as the power series
\begin{align*}
M(s,r)&=e^{-\tilde b s}\exp\Big\{\frac{\tilde b-r}a\big(1-e^{-as}\big)-\frac{\delta^2}{4a^3}\big(1-e^{-as}\big)^2\Big\}
\\&=e^{\frac{\tilde b-r-\frac{\delta^2}{4a^2}}a}\s_{n=0}^\infty e^{-(an+\tilde b)s}\s_{k=0}^{[n/2]}\frac{(-1)^{n-k}}{k!(n-2k)!}\Big(\frac{\tilde b-r-\frac{\delta^2}{2a^2}}a\Big)^{n-2k}\Big(\frac{\delta^2}{4a^3}\Big)^k 
\;,
\end{align*} 
From \cite[p.\ 295]{BorSa} we know that the density of $\tau_\zeta^z$ is given by
\begin{align}
f_{\tauzeta}(t) := \mP_z[\tau_\zeta\in\md t]=\frac{z}{\sqrt{2\pi}\sigma t^{3/2}}\exp\{-\frac{(z-(\mu-\zeta)t)^2}{2\sigma^2 t}\} \;.\label{density}
\end{align}
Moreover,
$$
\mP_z[\tau_\zeta = \infty] = 1-\exp\left(-\frac{(\mu-\zeta) z + |\mu-\zeta|z}{\sigma^2}\right)\;, \quad z \ge 0 \;.
$$
Letting
\[
\theta_n:=\frac{-\mu+\zeta-\sqrt{(\mu-\zeta)^2+2\sigma^2(an+\tilde b)}}{\sigma^2}\;,
\]
the power series representation of $\phi$ becomes
\begin{align*}
\phi(r,z)&=
 e^{\frac{\tilde b-r-\frac{\delta^2}{4a^2}}a}\s_{n=0}^\infty \mE_x\big[e^{-an\tau_\zeta^z-\tilde b \tau_\zeta^z}\big]\s_{k=0}^{[n/2]}\frac{(-1)^{n-k}}{k!(n-2k)!}\Big(\frac{\tilde b-r-\frac{\delta^2}{4a^2}}a\Big)^{n-2k}\Big(\frac{\delta^2}{4a^3}\Big)^k
\\&=e^{\frac{\tilde b-r-\frac{\delta^2}{4a^2}}a}\s_{n=0}^\infty e^{\theta_n z}\s_{k=0}^{[n/2]}\frac{(-1)^{n-k}}{k!(n-2k)!}\Big(\frac{\tilde b-r-\frac{\delta^2}{2a^2}}a\Big)^{n-2k}\Big(\frac{\delta^2}{4a^3}\Big)^k\;.
\end{align*}
Inserting $x-l$ instead of $z$ in the above expression, we get the condition specifying the curve $\theta$ (if such a curve exists):
\[
\phi(r,x-l)=e^{\frac{\tilde b-r-\frac{\delta^2}{4a^2}}a}\s_{n=0}^\infty e^{\theta_n (x-l)}\s_{k=0}^{[n/2]}\frac{(-1)^{n-k}}{k!(n-2k)!}\Big(\frac{\tilde b-r-\frac{\delta^2}{2a^2}}a\Big)^{n-2k}\Big(\frac{\delta^2}{4a^3}\Big)^k=1.
\]
The power series representation does not allow to show the existence and uniqueness of a curve $\theta\neq 0$ such that $\phi(r,x-l)> 1$ if $x>\theta(r,l)$ and $\phi(r,x-l)< 1$ if $x<\theta(r,l)$. We conclude that the approach consisting in finding a candidate return function and showing that this function solves the corresponding HJB equation cannot be applied here.
Although we will address this question in future research, in the present paper we will now tackle the problem using BSDEs. 

\subsection{BSDE approach}
Similar to the case of a deterministic time horizon, for an arbitrary strategy-independent $\left\{\mathcal{F}_t\right\}_{t\geq 0}$-stopping time $\tau$, we define the return function corresponding to some admissible strategy $c = \{c_s\}$ to be 
	\begin{align*}
		V^c(t,r,x)=&\mathbb{E}\left[\left(\int_{t\wedge \tau}^\tau e^{-U_{t\wedge\tau}^s}c_s\ds+e^{-U_{t\wedge\tau}^\tau}X^c_\tau\right) \one_{\{\tau <\infty\}} \middle|r_{t\wedge\tau}=r,X^c_{t\wedge\tau}=x.\right]\\
		&+\mathbb{E}\left[\left(\int_{t}^\infty e^{-U_{t}^s}c_s\ds\right)\one_{\{\tau=\infty\}} \middle|r_{t}=r,X^c_{t}=x\right],
	\end{align*}
	since in the event of $\left\{\tau=\infty\right\}$, the value $X^c_\infty$ is never reached. This is also consistent with the condition $\tilde{b}>0$ from Assumption \ref{assumption}, because then $$
	\lim_{T\to\infty}e^{-U_{t\wedge\tau}^T}X^c_T= 0\;.
	$$ Note that in the event $\{\tau<t\}$, any strategy $c$ comes to late since then $V^c(t,r,x)=x$ is already determined (and any $c$ is optimal, thus not interesting).	
	As $\tau$ does not depend on the strategy $c$, the problem above reduces to the case when $\tau<\infty$, as for the event $\left\{\tau=\infty\right\}$ it is clear that $c=\xi$ yields the best strategy. Hence, we have to find a strategy to optimize only
	\begin{align}\label{eq:optimize}
		\sup_{c\in \mathfrak{A}}\mathbb{E}\left[\left(\int_{t\wedge \tau}^\tau e^{-U_{t\wedge\tau}^s}c_s\ds+e^{-U_{t\wedge\tau}^\tau}X^c_\tau\right)\one_{\{\tau <\infty\}}\middle|r_{t\wedge\tau}=r,X^c_{t\wedge\tau}=x  \right]\;,\end{align}
	which we will tackle in the sequel using a BSDE approach.\bigskip
	
\noindent Within this section, for readability, we will use $\E[\ \cdot\ ]$ instead of $\E\left[\ \cdot\ \middle| r_{t\wedge\tau}=r,X^c_{t\wedge\tau}=x \right]$.\bigskip

\noindent The arguments of \cite[Section 4.2]{carmona}, \cite[Section 6.4.2]{pham} and \cite[Section 10.1.1]{touzi} may all be modified in an obvious way to use stopping times instead of deterministic ending times $T$. Like in Section \ref{sec:bsdesec1}, the solution for the optimal strategy can again be found via maximizing the generalized Hamiltonian 
\[
H(s,\mathtt{x},c,y,z)=(\mu-c)y+\mathrm{tr}\left(\begin{pmatrix}
	\sigma & 0
\end{pmatrix}^T z\right)+e^{-U_t^s}c\;,
\]
which one achieves by
\begin{align*}
	\hat c=\arg\max_cH(s,\mathtt{x},c,y,z)=\begin{cases}\xi,& e^{-U_{t\wedge\tau}^{s}}-y>0\;, \\0, & e^{-U_{t\wedge\tau}^{s}}-y\leq 0\;. \end{cases}
\end{align*}
We have to solve the corresponding BSDE for $Y$,
\[Y_s=e^{-U_{t\wedge\tau}^\tau}-\int_{s}^\tau Z_{u,1}\dW_u-\int_{s}^\tau Z_{u,2}\dB_u\;,\quad t\wedge\tau \leq s\leq\tau<\infty\;.\]
 For BSDEs with stopping times as time horizon see e.g. \cite[Section 3]{pardoux1998}, \cite{pardoux1999} or \cite[Section 5]{briand}. Again, in this very BSDE no generator appears and hence we end up with the conditional expectation $Y_s=\mathbb{E}\left[\exp\left(-\int_{t\wedge\tau}^\tau r_u \du\right)\middle|\mathcal{F}_s\right]=\mathbb{E}\left[\exp\left(-U_{t\wedge\tau}^\tau\right)\middle|\mathcal{F}_s\right]$, here with respect to the filtration $\left\{\mathcal{F}_s\right\}$ instead of $\{\mathcal{F}^B_s\}$. The reason is that, because of the presence of $\tau$, the terminal condition $e^{-U_{t\wedge\tau}^\tau}=e^{-\int_{t\wedge\tau}^\tau r_u \du}$ is not necessarily only $\mathcal{F}^B$-measurable. This is in particular the case for the stopping times $\tauzeta$ treated in the next section. \bigskip
 
With this solution for $Y$, it follows that if $t\wedge\tau\leq s\leq \tau<\infty$,
$$\hat{c}_s=\begin{cases}\xi\;,& e^{-U_{t\wedge\tau}^{s}}-Y_s>0, \\0\;, & e^{-U_{t\wedge\tau}^s}-Y_s\leq 0\;. \end{cases}$$
As before, to determine the barrier, we see that for $t\wedge\tau\leq s\leq \tau$, $$Y_{s}=\mathbb{E}\left[e^{-U_{t\wedge\tau}^\tau}\middle|\mathcal{F}_{s}\right]=\mathbb{E}\left[e^{-\int_{t\wedge\tau}^\tau r_u \du}\middle|\mathcal{F}_{s}\right]=e^{-\int_{t\wedge\tau}^{s} r_u \du}\mathbb{E}\left[e^{-\int_{s}^\tau r_u \du}\middle|\mathcal{F}_{s}\right],$$
and hence the barrier is attained if
\begin{align*}
e^{-\int_{t\wedge\tau}^sr_u \du}-e^{-\int_{t\wedge\tau}^{s} r_u \du}\mathbb{E}\left[e^{-\int_{s}^\tau r_u \du}\middle|\mathcal{F}_{s}\right]&=0\\
\Leftrightarrow\quad \mathbb{E}\left[e^{-\int_{s}^\tau r_u \du}\middle|\mathcal{F}_{s}\right]&=1\\
\Leftrightarrow\quad \mathbb{E}\left[e^{-U_{s}^\tau}\middle|\mathcal{F}_{s}\right]&=1\;.
\end{align*}

\subsubsection{Evaluating the strategy for the stopping times $\tauzeta$}

In the above subsection we found that the optimal strategy can be determined calculating the conditional expectation $\mathbb{E}\left[\exp\left(-U_{s}^\tau\right)\middle|\mathcal{F}_{s}\right] =
\mathbb{E}\left[\exp\left(-\int_{s}^{\tau} r_u \du\right)\middle|\mathcal{F}_{s}\right]$ 
on $\{t\leq s\leq\tau\}$. 
We will now do this explicitly for the stopping time $\tauzeta=\inf_{t \ge 0} \{ X_t^{H,\zeta} = 0 \}$.\bigskip 

Note that $X^{H,\zeta}$ is independent of $B$, and therefore $r$ and $X^{H,\zeta}$ are independent processes, and also $\tauzeta$ is independent of $r$. We will use this independence to show with $f_{\tauzeta}$, the density from \eqref{density}, we have
\begin{align}\label{eq:condexptau}
\mathbb{E}\left[e^{-\int_{s}^{\tauzeta} r_u \du}\middle|\mathcal{F}_{s}\right]\one_{\{\tauzeta>s\}}=\int_s^\infty \mathbb{E}\left[e^{-\int_{s}^{\tilde{t}} r_u \du}\middle|\mathcal{F}^B_{s}\right] f_{\tauzeta}(\tilde{t})\dtildet \; \one_{\{\tauzeta>s\}}\;.
\end{align}

To that end, first consider that $\mathcal{F}_s$ is generated by elements $A$ of the form $$
A=\left(B_{\cdot\wedge s}^{-1}(L_1)\cap W_{\cdot\wedge s}^{-1}(L_2)\right)\cup N\;,
$$
where $L_1,L_2$ are Borel subsets of $\mathcal{C}\left({[0,\infty)}\right)$, the space of continuous functions on $[0,\infty)$ endowed with the topology of uniform convergence on compacts, and $N\in \mathcal{N}$ is a null set. Our first observation is that for such $A$, 

\begin{align}\label{eq:condexpA}
&\mathbb{E}\left[\mathbb{E}\left[e^{-\int_{s}^{\tauzeta} r_u \du}\middle|\mathcal{F}_{s}\right]\one_{\{\tauzeta>s\}}\one_A\right]
=\mathbb{E}\left[\mathbb{E}\left[e^{-\int_{s}^{\tauzeta} r_u \du}\middle|\mathcal{F}_{s}\right]\one_{\{\tauzeta>s\}}\one_{B_{\cdot\wedge s}^{-1}(L_1)}\one_{W_{\cdot\wedge s}^{-1}(L_2)}\right]\nonumber\\
&=\E \left[ e^{-\int_{s}^{\tauzeta} r_u \du}\one_{\{\tauzeta>s\}}\one_{B_{\cdot\wedge s}^{-1}(L_1)}\one_{W_{\cdot\wedge s}^{-1}(L_2)}\right]\;,
\end{align}
where we used that $N$ is a null-set as well as the defining property of conditional expectation. We continue with 
\begin{align*}
&\E\left[ e^{-\int_{s}^{\tauzeta} r_u \du}\one_{\{\tauzeta>s\}}\one_{B_{\cdot\wedge s}^{-1}(L_1)}\one_{W_{\cdot\wedge s}^{-1}(L_2)}\right]=\\
&\quad\quad\int_s^\infty \E\left[e^{-\int_{s}^{\tilde{t}} r_u \du}\one_{B_{\cdot\wedge s}^{-1}(L_1)}\one_{W_{\cdot\wedge s}^{-1}(L_2)}\middle|\tauzeta=\tilde{t}\right] f_{\tauzeta}(\tilde{t})\dtildet\;,
\end{align*}
where we used regular conditional probabilities for the integrand, conditioning on $\tauzeta$. Now, using independence of $B$ and $r$ from $\tauzeta$, we get
\begin{align*}
&\int_s^\infty \E\left[e^{-\int_{s}^{\tilde{t}} r_u \du}\one_{B_{\cdot\wedge s}^{-1}(L_1)}\one_{W_{\cdot\wedge s}^{-1}(L_2)}\middle|\tauzeta=\tilde{t}\right] f_{\tauzeta}(\tilde{t})\dtildet\\
&=\int_s^\infty \E\left[e^{-\int_{s}^{\tilde{t}} r_u \du}\one_{B_{\cdot\wedge s}^{-1}(L_1)}\right]\E\left[\one_{W_{\cdot\wedge s}^{-1}(L_2)}\middle|\tauzeta=\tilde{t}\right] f_{\tauzeta}(\tilde{t})\dtildet\;.
\end{align*}
By definition of the conditional expectation with respect to $\mathcal{F}^B_s$, the last expression equals
\begin{align*}
\int_s^\infty \E\left[\E\left[e^{-\int_{s}^{\tilde{t}} r_u \du}\middle|\mathcal{F}^B_s\right]\one_{B_{\cdot\wedge s}^{-1}(L_1)}\right]\E\left[\one_{W_{\cdot\wedge s}^{-1}(L_2)}\middle|\tauzeta=\tilde{t}\right] f_{\tauzeta}(\tilde{t})\dtildet\;.
\end{align*}
From here, going back the same steps as before, and using that $\mathcal{F}^B_s\subseteq\mathcal{F}_s$, we find that the first term of \eqref{eq:condexpA} is given by
\begin{align*}
\mathbb{E}\left[\mathbb{E}\left[e^{-\int_{s}^{\tauzeta} r_u \du}\middle|\mathcal{F}^B_s\right]\one_{\{\tauzeta>s\}}\one_A\right], 
\end{align*}
from which we conclude, as $A$ was an arbitrary generator of $\mathcal{F}_s$, that 
$$
\mathbb{E}\left[e^{-\int_{s}^{\tauzeta} r_u \du}\middle|\mathcal{F}_{s}\right]\one_{\{\tauzeta>s\}}=\E\left[e^{-\int_{s}^{\tauzeta} r_u \du}\middle|\mathcal{F}^B_s\right]\one_{\{\tauzeta>s\}}\;.
$$
Performing now similar steps for some Borel set $L\subseteq \mathcal{C}({[0,\infty)})$ with the expectation
\begin{align*}
\E\left[\E\left[e^{-\int_{s}^{\tauzeta} r_u \du}\middle|\mathcal{F}^B_s\right]\one_{\{\tauzeta>s\}}\one_{B_{\cdot\wedge s}^{-1}(L)}\right],
\end{align*}
we get
\begin{align*}
\E \left[\E\left[e^{-\int_{s}^{\tauzeta} r_u \du}\middle|\mathcal{F}^B_s\right]\one_{\{\tauzeta>s\}}\one_{B_{\cdot\wedge s}^{-1}(L)}\right]=\int_s^\infty \E\left[\E\left[e^{-\int_{s}^{\tilde{t}} r_u \du}\middle|\mathcal{F}^B_s\right]\one_{B_{\cdot\wedge s}^{-1}(L)}\right]f_{\tauzeta}(\tilde{t})\dtildet\;,
\end{align*}
where we may exchange the integrals to end up with 
\begin{align*}
\E\left[\int_s^\infty \E\left[e^{-\int_{s}^{\tilde{t}} r_u \du}\middle|\mathcal{F}^B_s\right]f_{\tauzeta}(\tilde{t})\dtildet \; \one_{B_{\cdot\wedge s}^{-1}(L)}\right],
\end{align*}
from which \eqref{eq:condexptau}, i.e.
\begin{align*}
\mathbb{E}\left[e^{-\int_{s}^{\tauzeta} r_u \du}\middle|\mathcal{F}_{s}\right]\one_{\{\tauzeta>s\}}=\int_s^\infty\E\left[e^{-\int_{s}^{\tilde{t}} r_u \du}\middle|\mathcal{F}^B_s\right]f_{\tauzeta}(\tilde{t})\dtildet \; \one_{\{\tauzeta>s\}}
\end{align*}
follows. Applying \eqref{eq: condexplicit} to evaluate the conditional expectation in the last term, we get that
\begin{align}\label{eq:barrier_definer_bsde}
&\mathbb{E}\left[e^{-\int_{s}^{\tauzeta} r_u \du}\middle|\mathcal{F}_{s}\right]\one_{\{\tauzeta>s\}}\nonumber\\
&=\int_s^\infty e^{-\frac{r_s}{a}\left(1-e^{-a(\tilde{t}-s)}\right)-b(\tilde{t}-s)+\frac{b}{a}\left(1-e^{-a(\tilde{t}-s)}\right)+\frac{\delta^2}{2a^2}\int_s^{\tilde{t}}(1-e^{a(u-\tilde{t})})^2\du} f_{\tauzeta}(\tilde{t})\dtildet \; \one_{\{\tauzeta>s\}}\nonumber\\
&=\int_s^\infty M(\tilde{t}-s,r_s) f_{\tauzeta}(\tilde{t})\dtildet \; \one_{\{\tauzeta>s\}}
\end{align}
This expression \eqref{eq:barrier_definer_bsde} defines again the barrier, when being equal to 1.
In the special case of $s=t$, we get
\begin{align*}
&\int_t^\infty e^{-\frac{r}{a}\left(1-e^{-a(\tilde{t}-t)}\right)-b(\tilde{t}-t)+\frac{b}{a}\left(1-e^{-a(\tilde{t}-t)}\right)+\frac{\delta^2}{2a^2}\int_t^{\tilde{t}}(1-e^{a(u-\tilde{t})})^2\du} f_{\tauzeta}(\tilde{t})\dtildet \; \one_{\{\tauzeta>t\}}\\
&=\int_t^\infty M(\tilde{t}-t,r) f_{\tauzeta}(\tilde{t})\dtildet \; \one_{\{\tauzeta>t\}}=\left(M(\cdot,r)*(f_{\tauzeta}\cdot \one_{\{[t,\infty)\}})\right)\one_{\{\tauzeta>t\}}\;.
\end{align*}

\begin{figure*}[h]
	\centering
	\begin{subfigure}[h]{0.4\textwidth}
		\centering
		\scalebox{0.7}{
			\includegraphics[height=2.8in]{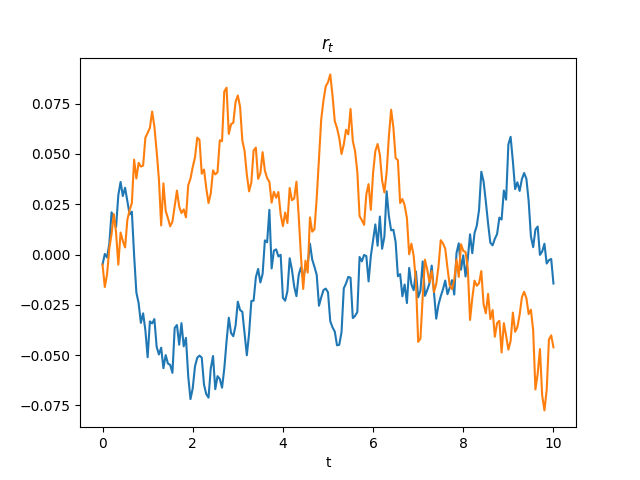}\label{fig:11}}
	\end{subfigure}%
	~ 
	\begin{subfigure}[h]{0.4\textwidth}
		\centering
		\scalebox{0.7}{
			\includegraphics[height=2.8in]{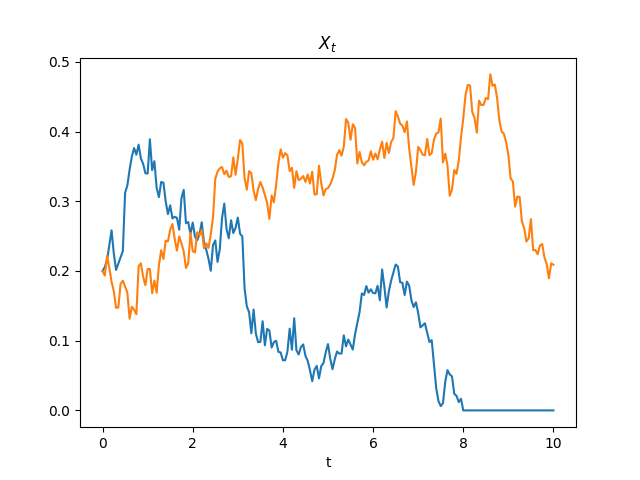}\label{fig:12}}
	\end{subfigure}
	\caption{Interest rate (left picture) and company wealth (right picture), where the blue path hit the stopping boundary.}
	\label{fig:rate_wealth}
\vspace*{\floatsep}
	\centering
	\begin{subfigure}[h]{0.4\textwidth}
		\centering
		\scalebox{0.7}{
			\includegraphics[height=2.8in]{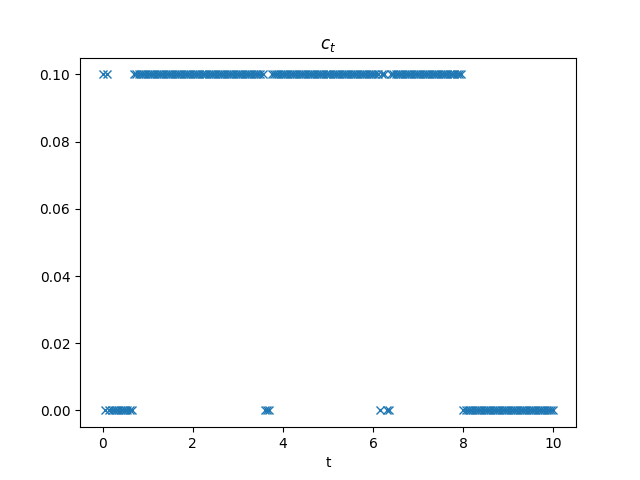}\label{fig:21}}
	\end{subfigure}%
	~ 
	\begin{subfigure}[h]{0.4\textwidth}
		\centering
		\scalebox{0.7}{
			\includegraphics[height=2.8in]{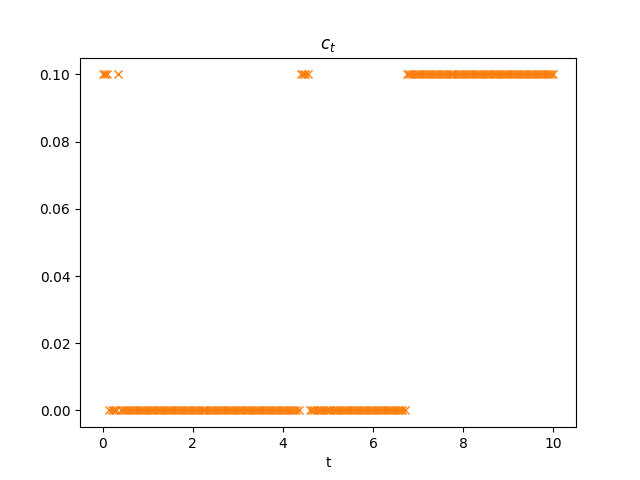}\label{fig:22}}
	\end{subfigure}
	\caption{The payout strategy $c_t$ for the paths of the surplus process in Figure \ref{fig:rate_wealth}.}
	\label{fig:strategy}
\end{figure*}

With the BSDE approach, we were able to find a solution to this control problem in form of the integral solution above. The exogenous stopping time now enters into the solution in form of the convolution $M(\cdot,r)*\left(f_{\tauzeta\cdot \one_{\{[t,\infty)\}}}\right)$. A numerical illustration can be found in Figure \ref{fig:rate_wealth} and Figure \ref{fig:strategy}.

\vspace{2pt} 



\newpage

\textbf{Acknowledgements}
\vspace{1em}

The research of Julia Eisenberg was funded by the Austrian Science
Fund (FWF), Project number V 603-N35.
\vspace{0.5em}

Stefan Kremsner was supported by the Austrian Science Fund (FWF): Project F5508-N26, which is part of the Special Research Program ''Quasi-Monte Carlo Methods: Theory and Applications''.


\end{document}